\documentclass[12pt]{article}
\usepackage[a4paper, total={6in, 8in}]{geometry}

\usepackage{amsmath, amsthm, amssymb, amsfonts, mathrsfs, enumerate}
\usepackage{amsrefs}
\usepackage{bbm}

\newtheorem{theorem}{Theorem}
\newtheorem{lemma}{Lemma}
\newtheorem{proposition}{Proposition}
\newtheorem{corollary}{Corollary}
\newtheorem{definition}{Definition}

\theoremstyle{remark}
\newtheorem*{remark}{Remark}

\title{On cyclicity in de Branges-Rovnyak spaces}
\author{Alex Bergman}                             
\date{\today}

\begin{document}

\maketitle

\begin{abstract}
  We study the problem of characterizing the cyclic vectors in de Branges-Rovnyak spaces. Based on a description of the invariant subspaces we show that the difficulty lies entirely in understanding the subspace $(aH^{2})^{\perp}$ and give a complete function theoretic description of the cyclic vectors in the case $\dim (aH^{2})^{\perp} < \infty$. Incidentally, this implies analogous results for certain generalized Dirichlet spaces $\mathcal{D}(\mu)$. Most of our attention is directed to the infinite case where we relate the cyclicity problem to describing the exposed points of $H^{1}$ and provide several sufficient conditions. A necessary condition based on the Aleksandrov-Clark measures of $b$ is also presented.
\end{abstract}

\section{Introduction}

This article is concerned with cyclic vectors in de Branges-Rovnyak spaces $\mathcal{H}(b)$. Let $\mathbb{D} = \left\{ z \in \mathbb{C} : \lvert z \rvert < 1 \right\}$ be the open unit disk in the complex plane and equip its boundary $\mathbb{T} = \left\{ \zeta \in \mathbb{C} : \lvert \zeta \rvert = 1\right\}$ with the normalized Lebesgue measure, $m$. For $0 < p < \infty$ the Hardy class $H^{p}$ is the set of analytic functions on $\mathbb{D}$ for which the norm (in the case $1 \leq p < \infty$)
\begin{equation*}
  \| f \|_{p}^{p} = \sup_{0 < r < 1} \int_{\mathbb{T}} \lvert f(r\zeta) \rvert^{p} dm(\zeta),
\end{equation*}
is finite. For $p = \infty$ we let $H^{\infty}$ be the class of bounded analytic functions on $\mathbb{D}$. We can in the usual way identify $H^{p}$ with a subspace of $L^{p} = L^{p}(\mathbb{T})$ via non-tangential limits. In this setting $H^{p}$ (with $p \geq 1$) consists of all $L^{p}$ functions, whose Fourier spectrum is contained in the nonnegative integers. Denote by $P_{+}$ the orthogonal projection from $L^{2}$ onto $H^{2}$. For a symbol $U \in L^{\infty}$ we define the Toeplitz operator $T_{U}f = P_{+}(Uf)$, $f \in H^{2}$.

Let $b$ be a nonconstant function in the unit ball of $H^{\infty}$. The de Branges-Rovnyak space $\mathcal{H}(b)$ is defined as the operator range of $(I - T_{b}T_{\overline{b}})^{1/2}$. Our main reference for the basic theory of $\mathcal{H}(b)$ is \cite{MR1289670}, see also the recent two-part monograph \cites{MR3497010, MR3617311}. The space $\mathcal{H}(b)$ is contractively contained inside the usual Hardy space, $H^{2}$. The theory of $\mathcal{H}(b)$ spaces splits into two cases depending on whether or not $b$ is an extreme point of the unit ball of $H^{\infty}$. De Branges-Rovnyak spaces have been studied intensively. In particular, the problem of smooth approximation has received attention in the case of extreme $b$ in recent years, see for example \cites{MR3947672, MR3693507, MR4521737, https://doi.org/10.48550/arxiv.2111.14112, https://doi.org/10.48550/arxiv.2108.08625}. In the non-extreme case polynomials form a dense subset and so the question of smooth approximation is trivial. However, non-extreme de Branges-Rovnyak spaces are forward shift invariant and hence the question of characterizing the cyclic vectors is meaningful in this case. For a function $f$ in a Hilbert space of analytic functions $H$ invariant under the forward shift operator, $M_{z}f = zf$, the cyclic subspace generated by $f$ is defined as the closure of the linear span of polynomial multiples of $f$. It will be denoted $[f]$ and a function $f$ is called cyclic (for $H$) if $[f] = H$. Before stating our goals we will need some definitions.

If $\mathcal{H}(b)$ is invariant under the forward shift operator defined by $M_{z}f = zf$ it is known that $b$ is a non-extreme point of the unit ball of $H^{\infty}$, that is
\begin{equation*}
  \int_{\mathbb{T}} \log (1 - |b(\zeta)|) dm(\zeta) > -\infty.
\end{equation*}
Thus the problem of classifying the cyclic vectors for $\mathcal{H}(b)$ is meaningful only in the non-extreme case.

For a non-extreme $b$ we define the unique outer function $a(0) > 0$ satisfying $\lvert b \rvert^{2} + \lvert a \rvert^{2} = 1$ a.e. on $\mathbb{T}$. The space $aH^{2} = \left\{ ah : h \in H^{2} \right\}$ is contractively contained inside $\mathcal{H}(b)$. The problem of classifying the cyclic vectors in de Branges-Rovnyak spaces was raised by Fricain in \cite{MR3589677}. Also in \cite{MR3309352} the cyclic vectors in the case $b = (1+z)/2$ were determined. In Section \ref{sec:suff} we generalize this considerably by giving a complete function theoretic characterization of the cyclic vectors in the case $\dim(aH^{2})^{\perp} < \infty$ (the symbol $(aH^{2})^{\perp}$ denotes the orthogonal complement of $aH^{2}$ in $\mathcal{H}(b)$).
    \begin{theorem}\label{thm:finite_defect}
      Let $\mathcal{H}(b)$ be a non-extreme de Branges-Rovnyak space and suppose $\dim(aH^{2})^{\perp} < \infty$. Denote by $\overline{\lambda}_{1}, \overline{\lambda}_{2}, ..., \overline{\lambda}_{s}$ the eigenvalues of $M_{z}^{*}$ restricted to $(aH^{2})^{\perp}$. Then
      \begin{enumerate}[(i)]
        \item each $\overline{\lambda}_{j}$ lies on $\mathbb{T}$,
        \item every function $h \in \mathcal{H}(b)$ extends non-tangentially to $\lambda_{j}$,
        \item $f \in \mathcal{H}(b)$ is cyclic if and only if $f$ is outer and $f(\lambda_{j}) \neq 0$, for all $j = 1, 2, ..., s$.
        \end{enumerate}
      \end{theorem}

\begin{remark}
    Sarason has shown that $\dim(aH^{2})^{\perp} < \infty$ if and only if the function $\phi = a/(1-b)$ is of the form $\phi = Fp$, where $F \in H^{2}$ is an outer function, such that $F^{2}$ is an exposed point of $H^{1}$ (exposed points will be defined shortly) and $p$ is a polynomial with all of its zeros on the unit circle, see (X-17) in \cite{MR1289670}.
    \end{remark}
    Thus Theorem \ref{thm:finite_defect} contains the case of rational $b$ often considered in the literature as a special case.
    
    Our main efforts will go towards the case of infinite codimension, $\dim(aH^{2})^{\perp} = \infty$. This case seems drastically more difficult than the finite case. In light of this, we settle for providing a necessary and several sufficient conditions.

  Using a recent description of the invariant subspaces of $\mathcal{H}(b)$ in \cite{MR3947672} we prove the following basic condition for cyclicity valid for any non-extreme de Branges-Rovnyak space.

For a Borel set $E \subset \mathbb{T}$ denote by $\mathbbm{1}_{E}$ the function that is $1$ on $E$ and $0$ on $\mathbb{T} \setminus E$.
\begin{theorem}\label{thm:thmA}
  Let $\mathcal{H}(b)$ be a non-extreme de Branges-Rovnyak space and $f \in \mathcal{H}(b)$ an outer function. Suppose there exists Borel sets $E, F \subset \mathbb{T}$, such that
  \begin{enumerate}[(i)]
    \item $E \cup F = \mathbb{T}$,
    \item $a^{-1}\mathbbm{1}_{E} \in L^{2}$ and $f^{-1}\mathbbm{1}_{F} \in L^{\infty}$.
  \end{enumerate}
  Then $f$ is cyclic.
\end{theorem}
More precise results require a detailed analysis of the space $(aH^{2})^{\perp}$. The case $\dim(aH^{2})^{\perp} = \infty$ has appeared indirectly in a conjecture of Sarason on the exposed points of the unit ball of $H^{1}$ (see Chapter X of \cite{MR1289670}) and in the negative answer to that conjecture in \cites{MR1889082, MR1734331, MR1258029}. As we shall see there is an intimate connection between the cyclic vectors in $\mathcal{H}(b)$ and the exposed points of the unit ball of $H^{1}$.

To motivate what is to come we briefly describe exposed points of the unit ball of $H^{1}$. For a convex set $K$ in a linear space $X$ a point $x \in K$ is called an exposed point of $K$ if there exists a real linear functional $\ell$, such that $\ell(x) > \ell(k)$, for all $k \in K \setminus \left\{ x \right\}$. Exposed points of the unit ball of $H^{1}$ will be called exposed points if no confusion can arise. Exposed points are extreme points and so if $f \in H^{1}$ is an exposed point of the unit ball it is an outer function \cite{MR98981}. In $H^{1}$ there is an alternative characterization in terms of the argument of the boundary function. An outer function $f \in H^{1}$ of unit norm is an exposed point if and only if the only functions in $H^{1}$ with the same argument a.e. on $\mathbb{T}$ are positive multiples of $f$, here the argument function is the principal branch $\arg (z) \in [0,2\pi)$. Also, an outer function of $f \in H^{1}$ of unit norm is exposed if and only if the Toeplitz operator $T_{\overline{f}/f}$ has trivial kernel, see \cite{MR1289670}. There is an extensive literature on exposed points in $H^{1}$, see for example \cites{MR1207406, MR1038352, MR1223178, MR1734331, MR1258029}. Despite this, there is no characterization of exposed points based on the modulus of the function on the unit circle.

Exposed points are connected to cyclicity in $\mathcal{H}(b)$ in the following way: the function $\phi = a/(1-b)$ is an outer function and, after suitable normalization of $b$, it is of unit norm. For such functions we shall consider a set $\sigma(\phi) \subset \mathbb{T}$ measuring in some sense how far away $\phi^{2}$ is from being an exposed point (For the rigorous definition see Definition \ref{def:specf} in Section \ref{sec:spec}). A heuristic principle is that an outer function $f \in \mathcal{H}(b)$ is cyclic if it is ``not too small'' on $\sigma(\phi)$. In Section \ref{sec:suff} we prove the following concrete manifestation of this principle.
\begin{theorem}\label{thm:thmB}
  Let $\mathcal{H}(b)$ be a non-extreme de Branges-Rovnyak space, $b(0)=0$, and set $\phi = a/(1-b)$. Without loss of generality we may assume $\phi$ is of unit norm. Let, in addition, $f \in \mathcal{H}(b)$ be an outer function. Suppose that for each point $\zeta \in \sigma(\phi)$ there exists an open arc $\zeta \in I_{\zeta} \subset \mathbb{T}$ and number $\eta_{\zeta} > 0$ with $\lvert f \rvert > \eta_{\zeta}$ a.e. on $I_{\zeta}$. Then $f$ is cyclic.
\end{theorem}
We shall also prove a sharper version of the above theorem which depends only on the behavior of $f$ at each point in $\sigma(\phi)$ (and not in a neighborhood). Before stating the result we will comment on our method. For an outer function $\phi$ consider the kernel of the Toeplitz operator $T_{\overline{\phi}/\phi}$ and let $J_{\phi} = \phi^{-1} \ker (T_{\overline{\phi}/\phi})$. Functions in $J_{\phi}$ possess the remarkable property of analytic pseudocontinuation, that is: $J_{\phi}$ consists of functions $f$ analytic in $\mathbb{C}_{\infty} \setminus \mathbb{T}$, such that their non-tangential limits from inside and outside the unit disk coincide a.e. We shall realize the space $(aH^{2})^{\perp}$ as a space of normalized Cauchy transforms of functions in $J_{\phi}$. Thus the difference in the case $\dim (aH^{2})^{\perp} < \infty$ and the infinite case reflects the difference between describing finite and infinite dimensional Toeplitz kernels. We proceed to show that the cyclicity problem in $\mathcal{H}(b)$ is intimately related to the question of analytic continuation in $J_{\phi}$. For the next theorem denote by $V$ the operator
\begin{equation*}
  Vh(z) = (1-b(z)) \int_{\mathbb{T}} \frac{h(\zeta)\lvert\phi(\zeta)\rvert^{2}dm(\zeta)}{1-z\overline{\zeta}} \text{, } h \in L^{2}(\lvert \phi \rvert^{2}dm).
\end{equation*}
The following sharpening of Theorem \ref{thm:thmB} is proved in Section \ref{sec:suff}.
\begin{theorem}\label{thm:thmC}
    Let $\mathcal{H}(b)$ be a non-extreme de Branges-Rovnyak space, $b(0) = 0$, and set $\phi = a/(1-b)$. Without loss of generality we may assume $\phi$ is of unit norm. Let, in addition, $g \in H^{\infty}$ and write $f = Vg$. Denote by $\theta$ the inner factor of $f$ and set $F = f/\theta$. If $\sigma(F) \cap \sigma(\phi) = \varnothing$, then $F$ is cyclic.
\end{theorem}
Note that it is a part of the conclusion of Theorem \ref{thm:thmC} that $F = f/\theta \in \mathcal{H}(b)$. Theorem \ref{thm:thmC} is more precise than Theorem \ref{thm:thmB} since it requires information only at each point of $\sigma(\phi)$ and not in a neighborhood of every point, however, it applies to fewer functions since it requires $f = Vg$, with $g \in H^{\infty}$ instead of merely $g \in H^{2}/\phi$.

We end with a section on examples and applications. In particular, we use our results to give necessary and sufficient conditions for the cyclicity of $b$ and the kernel elements $k_{\lambda}^{b}$. We also describe the cyclic vectors in certain generalized Dirichlet spaces. Finally, we consider the case $b(z) = (1+\theta)/2$, where $\theta$ is a non-constant inner function. These give examples for which $\text{dim}(aH^{2})^{\perp} = \infty$, but we can still give a necessary and sufficient function theoretic condition for cyclicity.

\subsection*{Acknowledgements}
    The author expresses his deep gratitude to Alexandru Aleman for having shared the problems investigated here and for several helpful discussions.

\section{Preliminaries}

\subsection{The space $\mathcal{H}(b)$}

To each bounded analytic function $b : \mathbb{D} \to \mathbb{D}$ we associate the Hilbert space of analytic functions, $\mathcal{H}(b)$, defined as the range $(I-T_{b}T_{b}^{*})^{1/2}H^{2}$ with induced inner product
\begin{equation*}
  \langle (I-T_{b}T_{b}^{*})^{1/2}f, (I-T_{b}T_{b}^{*})^{1/2}g \rangle_{\mathcal{H}(b)} = \langle f, g \rangle_{2},
\end{equation*}
for $f,g \perp \ker((I-T_{b}T_{b}^{*})^{1/2})$. Equivalently $\mathcal{H}(b)$ can be seen as the reproducing kernel Hilbert space associated with the reproducing kernel $k_{\lambda}^{b} = (1-\overline{b(\lambda)}b(z))/(1-\overline{\lambda}z)$, $\lambda \in \mathbb{D}$. As we shall ultimately be interested in cyclic vectors of the forward shift, $M_{z}f = zf$, we shall confine our attention to those $b$ satisfying $z\mathcal{H}(b) \subset \mathcal{H}(b)$. It can be shown (see \cite{MR1289670}) that this happens if and only if $b$ is a non-extreme point of the unit ball of $H^{\infty}$, by \cite{MR98981}, this happens if and only if $b$ satisfies
\begin{equation*}
  \int_{\mathbb{T}} \log (1 - \lvert b \rvert) dm > - \infty.
\end{equation*}
We shall also, in an attempt to simplify formulas, assume $b(0) = 0$. Since $b$ is non-extreme we can introduce the unique outer function $a$ satisfying $\lvert a \rvert^{2} + \lvert b \rvert^{2} = 1$, a.e. on $\mathbb{T}$ and $a(0) > 0$. Now for $\alpha \in \mathbb{T}$ the function $(1+ \overline{\alpha}b)/(1-\overline{\alpha}b)$ has nonnegative real part in the unit disk and hence can be represented via the Herglotz integral formula
\begin{equation}\label{eq:herglotzformula}
  \frac{1+\overline{\alpha}b}{1-\overline{\alpha}b} = \int_{\mathbb{T}} \frac{\zeta + z}{\zeta - z} d\mu_{\alpha}(\zeta),
\end{equation}
for some Borel probability measure $\mu_{\alpha}$. The measures $\mu_{\alpha}$ are the Aleksandrov-Clark measures of $b$. Taking real parts and using properties of the Poisson kernel we see that $(1 - \lvert b \rvert^{2})/\lvert \alpha - b \rvert^{2}$ is the Radon-Nikodym derivative of the absolutely continuous part of $\mu_{\alpha}$ with respect to normalized Lebesgue measure. It follows that the measure $\mu_{\alpha}$ has the following decomposition in terms of its absolutely continuous and singular parts
\begin{equation*}
    d\mu_{\alpha} = \lvert \phi_{\alpha} \rvert^{2}dm + d(\mu_{\alpha})_{s},
\end{equation*}
where $\phi_{\alpha} = a/(1-\overline{\alpha}b) \in H^{2}$ is an outer function. Let $P^{2}(\mu_{\alpha})$ denote the closure of the analytic polynomials, $\text{Span} (\left\{ z^{n} : n \geq 0 \right\})$, in $L^{2}(\mu_{\alpha})$. Then $P^{2}(\mu_{\alpha})$ decomposes as
\begin{equation}\label{eq:splitting}
    P^{2}(\mu_{\alpha}) = \frac{H^{2}}{\phi_{\alpha}} \oplus L^{2}((\mu_{\alpha})_{s}),
\end{equation}
where $H^{2}/\phi_{\alpha} = \left\{ f/\phi_{\alpha} : f \in H^{2} \right\}$. Since $\lvert b(\zeta) \rvert < 1$ for almost every $\zeta \in \mathbb{T}$ an application of Fubini's Theorem
\begin{equation*}
  \begin{split}
    \int_{\mathbb{T}} (\mu_{\alpha})_{a}(\mathbb{T})dm(\alpha) = \int_{\mathbb{T}} \int_{\mathbb{T}} \frac{1-\lvert b(\zeta) \rvert^{2}}{\lvert \alpha - b(\zeta) \rvert}dm(\zeta)dm(\alpha) \\ = \int_{\mathbb{T}} \int_{\mathbb{T}} \frac{1-\lvert b(\zeta) \rvert^{2}}{\lvert \alpha - b(\zeta) \rvert}dm(\alpha)dm(\zeta) = 1,
  \end{split}
\end{equation*}
shows that for non-extreme $b$ the measures $\mu_{\alpha}$ are absolutely continuous for a.e. $\alpha \in \mathbb{T}$. By replacing $b$ by $\bar{\alpha}b$ if necessary we may assume $\mu = \mu_{1}$ is absolutely continuous of the form $d\mu = \lvert \phi \rvert^{2}dm$, with $\phi = a/(1-b)$. In particular we are free to assume that equation \eqref{eq:splitting} reduces to $P^{2}(\mu_{1}) = H^{2}/\phi$. For a Borel measure $\nu$ on $\mathbb{T}$ we introduce the Cauchy transform
\begin{equation*}
    C_{\nu}(z) = \int_{\mathbb{T}} \frac{d\nu(\zeta)}{1-z\overline{\zeta}},
\end{equation*}
and the operator
\begin{equation*}
  C_{\nu}h(z) = \int_{\mathbb{T}} \frac{h(\zeta)d\nu(\zeta)}{1-z\overline{\zeta}}, h \in L^{2}(\nu).
\end{equation*}
Associated to the finite Borel measure, $d\mu_{\alpha} = \lvert \phi_{\alpha} \rvert^{2} dm + d(\mu_{\alpha})_{s}$, we consider the normalized Cauchy transform
\begin{equation}\label{eq:unitarymap}
  V_{\alpha}h(z) = \frac{C_{\mu_{\alpha}}h(z)}{C_{\mu_{\alpha}}(z)} = (1-\overline{\alpha}b(z)) \int_{\mathbb{T}} \frac{hd\mu_{\alpha}(\zeta)}{1-z\overline{\zeta}},  h \in L^{2}(\mu_{\alpha}).
\end{equation}
The next result is standard and can be found in \cite{MR1289670}. We sketch a proof for the convenience of the reader.

\begin{proposition}\label{prop:unitarymap}
  Let $\mathcal{H}(b)$ be a de Branges-Rovnyak space and $\mu_{\alpha}$ an Aleksandrov-Clark measure of $b$. Then the map $V_{\alpha}$ from \eqref{eq:unitarymap} is a unitary operator $V_{\alpha} : P^{2}(\mu_{\alpha}) \to \mathcal{H}(b)$. Conversely, suppose $\phi \in H^{2}$ is an outer function of unit norm. Then there exists a unique non-extreme $b : \mathbb{D} \to \mathbb{D}$ satisfying $b(0) = 0$ and $V_{1}(H^{2}/\phi) = \mathcal{H}(b)$.
  \begin{proof}
      Let $k_{\lambda} = 1/(1-\overline{\lambda}z)$ be the Cauchy kernel. A computation based on equation \eqref{eq:herglotzformula} gives
      \begin{equation*}
        \begin{split}
          (1-\alpha \overline{b(\lambda)})V_{\alpha}k_{\lambda} = (1-\alpha \overline{b(\lambda)})(1-\overline{\alpha} b(z)) \int_{\mathbb{T}} \frac{1}{1-\overline{\lambda}\zeta}\frac{1}{1-z \overline{\zeta}}d\mu_{\alpha}(\zeta) \\ = \frac{(1-\alpha \overline{b(\lambda)})(1-\overline{\alpha} b(z))}{2(1-\overline{\lambda}z)} \int_{\mathbb{T}} \overline{\left( \frac{\zeta + \lambda}{\zeta -\lambda} \right)} + \frac{\zeta + z}{\zeta - z}d\mu_{\alpha}(\zeta) \\ = \frac{(1-\alpha \overline{b(\lambda)})(1-\overline{\alpha} b(z))}{2(1-\overline{\lambda}z)} \left( \frac{1+\alpha \overline{b(\lambda)}}{1-\alpha \overline{b(\lambda)}} + \frac{1+\overline{\alpha} b(z)}{1+\overline{\alpha} b(z)}\right) = k_{\lambda}^{b}(z).
        \end{split}
      \end{equation*}
      Also,
      \begin{equation*}
          \| k_{\lambda}^{b} \|_{b}^{2} = k_{\lambda}^{b}(\lambda) = \frac{1-\lvert b(\lambda) \rvert^{2}}{1-\lvert \lambda \rvert^{2}},
      \end{equation*}
      and by equation \eqref{eq:herglotzformula}
      \begin{equation*}
          \| (1-\alpha \overline{b(\lambda)})k_{\lambda} \|_{L^{2}(\mu_{\alpha})}^{2} = \frac{\lvert 1-\alpha \overline{b(\lambda)} \rvert^{2}}{1-\lvert \lambda \rvert^{2}}\frac{1-\lvert b(\lambda) \rvert^{2}}{\lvert 1 - \overline{\alpha}b(\lambda) \rvert^{2}} = \frac{1-\lvert b(\lambda) \rvert^{2}}{1-\lvert \lambda \rvert^{2}}.
      \end{equation*}
      Since $\left\{ (1-\alpha \overline{b(\lambda)})k_{\lambda} \right\}_{\lambda \in \mathbb{D}}$ is complete (by complete we mean it has dense linear span) in $P^{2}(\mu_{\alpha})$ and $\left\{ k_{\lambda}^{b} \right\}_{\lambda \in \mathbb{D}}$ is complete in $\mathcal{H}(b)$ the above computations and a simple limiting argument show that $V_{\alpha}$ is unitary. For the converse, it is sufficient to define $b$ via the identity
      \begin{equation*}
          \frac{1+b(z)}{1-b(z)} = \int_{\mathbb{T}} \frac{\zeta + z}{\zeta -z}\lvert \phi \rvert^{2}dm(\zeta).
      \end{equation*}
      The desired properties follow easily.
  \end{proof}
\end{proposition}
\begin{remark}
    We remark that the first part of of Proposition \ref{prop:unitarymap} does not require $b$ to be non-extreme, nor does it require $b(0) = 0$. In particular if $\theta$ is an inner function with Aleksandrov-Clark measure $\sigma$ at $\alpha \in \mathbb{T}$, we have $V_{\alpha}L^{2}(\sigma) = K_{\theta}$, where $K_{\theta} = H^{2} \cap \theta \overline{H^{2}_{0}}$ is the usual model space (here $\overline{H^{2}_{0}}$ is the orthogonal complement of $H^{2}$ in $L^{2}$). We shall need this in Section \ref{sec:ex}.
\end{remark}
Thus there is a one-to-one correspondence between outer functions of unit norm in $H^{2}$ and non-extreme $\mathcal{H}(b)$ spaces (with $b(0)=0$ and $\mu_{1}$ absolutely continuous). Thus in the sequel, we may speak of the $\mathcal{H}(b)$ space generated by an outer function $\phi \in H^{2}$ of unit norm.

We end this section by recalling a Theorem of Poltoratski that we will need in the continuation, see Theorem 2.7. in \cite{MR1223178}.

\begin{theorem}\label{thm:poltoratski}
    Let $\sigma$ be a finite positive Borel measure on $\mathbb{T}$ and denote its singular part by $(\sigma)_{s}$. Let $V$ be the operator
    \begin{equation*}
        Vh(z) = C_{\sigma}(z)^{-1}\int_{\mathbb{T}} \frac{h(\zeta)d\sigma(\zeta)}{1-z\overline{\zeta}} \text{, } h \in L^{2}(\sigma).
    \end{equation*}
    Suppose $h \in L^{2}(\sigma)$, then $Vh$ converges non-tangentially to $h$ $(\sigma)_{s}$-a.e.
\end{theorem}

\subsection{The role of $aH^{2}$}\label{sec:ah2}

For a Hilbert space $H$ and a bounded linear operator $T : H \to H$. An element $x \in H$ is called cyclic for $T$ if $\left\{ T^{n}x : n \geq 0 \right\}$ has dense linear span. For a Hilbert space $X$ of analytic functions invariant under multiplication by independent variable we shall denote by $[f]$ the closure of the linear span of $\left\{ z^{n}f : n \geq 0\right\}$. The linear subspace $[f]$ is called the cyclic subspace generated by $f$. In particular $f$ is a cyclic vector for $M_{z}f = zf$ if and only if $[f] = X$. A cyclic vector for $M_{z}$ will simply be called cyclic.

Since polynomials are dense in every non-extreme $\mathcal{H}(b)$ space it is necessary and sufficient for cyclicity that $1 \in [f]$. From the contractive inclusion $\mathcal{H}(b) \subset H^{2}$ we see that for any polynomial $p$ and $f \in \mathcal{H}(b)$
\begin{equation*}
  \| 1 - pf \|_{b} \geq \| 1 -pf \|_{2},
\end{equation*}
and so any cyclic vector in $\mathcal{H}(b)$ must be cyclic for $H^{2}$ and hence an outer function. Returning now to the unique outer function $a$ satisfying $\lvert a \rvert^{2} + \lvert b \rvert^{2} = 1$ a.e. on $\mathbb{T}$ and $a(0) > 0$ we may consider the space $aH^{2}$ with norm $\| af\|_{a} = \| f \|_{2}$. It follows from the operator inequality $T_{a}T_{a}^{*} \leq I - T_{b}T_{b}^{*}$ that $aH^{2}$ is contained contractively in $\mathcal{H}(b)$. Thus $\mathcal{H}(b)$ splits as
\begin{equation*}
  \mathcal{H}(b) = \text{clos}(aH^{2}) \oplus (aH^{2})^{\perp},
\end{equation*}
where $\text{clos}(aH^{2})$ denotes the closure of $aH^{2}$ in $\mathcal{H}(b)$ and $(aH^{2})^{\perp}$ denotes the orthogonal complement of $aH^{2}$ in $\mathcal{H}(b)$. As we shall see in a moment $\text{clos}(aH^{2})$ is fairly tame and presents little difficulty in terms of the cyclicity problem. Indeed in a sense it behaves very much like the usual space $H^{2}$. To see this we shall need the following result, see (IV-I) in \cite{MR1289670}.

\begin{proposition}\label{thm:J_model}
  Let $b$ be a non-extreme point of the unit ball of $H^{\infty}$. A function $f \in H^{2}$ lies in $\mathcal{H}(b)$ if and only if there exists $f_{1} \in H^{2}$ satisfying
    \begin{equation*}
      T_{\overline{b}}f + T_{\overline{a}}f_{1} = 0,
    \end{equation*}
and in this case $\| f \|_{b}^{2} = \| f \|_{2}^{2} + \| f_{1} \|_{2}^{2}$. We denote by  $J : \mathcal{H}(b) \to H^{2} \oplus H^{2}$ the isometry $Jf = (f,f_{1})$.
\end{proposition}
In \cite{MR3947672} it was shown that any closed invariant subspace $\mathcal{M}$ of $\mathcal{H}(b)$ is of the form
\begin{equation}\label{eq:invariantsubspaces}
  \mathcal{M} = \left\{ g \in \mathcal{H}(b) : \frac{g}{\psi}, \frac{g}{\psi}\psi_{1} \in H^{2} \right\},
\end{equation}
for some $\psi \in \mathcal{H}(b)$ and $J\psi = (\psi, \psi_{1})$. Since the right-hand side is invariant under multipliers of $\mathcal{H}(b)$ we have $U [f] \subset [f]$ for all multipliers $U$ of $\mathcal{H}(b)$.

The next result is probably well-known to experts, however, we could not find it in the literature. It concerns the so-called $(F)$ property of $\mathcal{H}(b)$. The proof is short and relies on a classical Theorem of S.A. Vinogradov on division of Cauchy integrals by inner functions \cite{MR0586560}.

\begin{lemma}\label{lemma:divisionofinner}
  Let $f \in \mathcal{H}(b)$ and denote the inner factor of $f$ by $\theta$. Then $f/\theta \in \mathcal{H}(b)$.
  \begin{proof}
    We may assume the Aleksandrov-Clark measure $\mu_{1}$ is absolutely continuous. Let $\phi = a/(1-b)$ and $g \in H^{2}/\phi$ be the unique function, such that $f = Vg$, where $V = V_{1}$ is the unitary map from Proposition \ref{prop:unitarymap}. Denote by $\mu = \lvert \phi \rvert^{2}m$ and recall that for a Borel measure $\nu$ we let $C_{\nu}$ denote the Cauchy transform of $\nu$ and for $h \in L^{1}(\nu)$
    \begin{equation*}
      C_{\nu}h(z) = \int_{\mathbb{T}} \frac{h(\zeta)d\nu(\zeta)}{1-z\overline{\zeta}}. 
    \end{equation*}
    From the equality
    \begin{equation*}
      f = Vg = (1-b)C_{\mu}g,
    \end{equation*}
    and using that $(1-b)$ has no inner factor we see that $\theta^{-1}C_{\mu}g \in H^{p}$, $0 < p < 1$. Thus applying the variant of Vinogradov's theorem contained in Theorem 3.4. of \cite{MR1223178} we have
    \begin{equation*}
      f/\theta = (1-b)C_{\overline{\theta}\mu}g = (1-b)C_{\mu}(\frac{T_{\theta}^{*}(\phi g)}{\phi}),
    \end{equation*}
    where $T_{\theta}$ is the Toeplitz operator with symbol $\theta$. Since $T_{\theta}^{*}(\phi g)/\phi \in H^{2}/\phi$ we have $f/\theta \in \mathcal{H}(b)$.
  \end{proof}
\end{lemma}

\begin{remark}
    The author was informed by Emmanuel Fricain that this appears in \cite{MR3617311} as Theorem 18.16 and Corollary 18.17.
\end{remark}

We are now ready to prove that cyclic subspaces of $\mathcal{H}(b)$ preserve inner factors and that for outer functions $aH^{2} \subset [f]$.

\begin{proposition}\label{prop:ah2}
  Let $\mathcal{H}(b)$ be a non-extreme de Branges-Rovnyak space and $f \in \mathcal{H}(b)$. Denote by $\theta$ the inner factor of $f$. Then $a\theta H^{2} \subset [f] \subset \theta \mathcal{H}(b)$. In particular if $f$ is outer $aH^{2} \subset [f]$.
  \begin{proof}
    Since $\mathcal{H}(b) \subset H^{2}$ and $aH^{2} \subset \mathcal{H}(b)$ the function $a$ is a multiplier of $\mathcal{H}(b)$. Thus by the discussion following equation \eqref{eq:invariantsubspaces} we see that $a[f] \subset [f]$. For any $h \in H^{2}$ and polynomials $p$ and $q$ we have
    \begin{equation*}
      \| a\theta h - pf\|_{b} \leq \| a\theta h - aqf \|_{b} + \| aqf - pf\|_{b} \leq \| \theta (h - qf\theta^{-1})\|_{2} + \| aqf - pf\|_{b}.
    \end{equation*}
    For $\epsilon > 0$ and using that $f\theta^{-1}$ is an outer function we can choose $q$, such that $\| h - qf\theta^{-1}\|_{2} < \epsilon$. Since $aqf \in [f]$ we can choose $p$, such that $\| aqf - pf\|_{b} < \epsilon$. Hence $a\theta h \in [f]$. Also if $p_{n}f \to g \in \mathcal{H}(b)$ for some sequence of polynomials $p_{n}$ we also have $p_{n}f \to g$ in $H^{2}$ and so $g \theta^{-1} \in H^{2}$. An application of Lemma \ref{lemma:divisionofinner} gives $g\theta^{-1} \in \mathcal{H}(b)$.
  \end{proof}
\end{proposition}

The above result implies the following classification of cyclic vectors in the case when $aH^{2}$ is dense in $\mathcal{H}(b)$.

\begin{corollary}\label{cor:phi_exposed}
  Let $\mathcal{H}(b)$ be a non-extreme de Branges-Rovnyak space and suppose $aH^{2} \subset \mathcal{H}(b)$ is dense. Then $f \in \mathcal{H}(b)$ is cyclic if and only if $f$ is outer.
\end{corollary}

We note that the density of $aH^{2}$ in $\mathcal{H}(b)$ occurs if and only if $\phi^{2} = (a/(1-b))^{2}$ is an exposed point of the unit ball of $H^{1}$. Note also that a precise description of the cyclic vectors in the case $\dim (aH^{2})^{\perp} < \infty$ can also be obtained from the above result. We defer the proof of this to Section \ref{sec:suff}.

\section{A model for $M_{z}^{*}$ on $(aH^{2})^{\perp}$}\label{sec:model}

In the previous section we have seen that the problem of cyclicity relies entirely on understanding the subspace $(aH^{2})^{\perp} \subset \mathcal{H}(b)$. Motivated by this we introduce a normalized Cauchy transform model for $(aH^{2})^{\perp}$. We shall consider a collection of spaces that have appeared in \cite{MR1734331} in connection with the problem of describing the exposed points of the unit ball of $H^{1}$. For an outer function $\phi \in H^{2}$ consider the weight $w =\lvert \phi \rvert^{2} \geq 0$, $w \in L^{1}$, and $\log w \in L^{1}$. The closure of the analytic polynomials,  $\text{Span} (\left\{ z^{n} : n \geq 0 \right\})$, in the space $L^{2}(w)$ coincides with $H^{2}/\phi$. Similarly for the anti-analytic polynomials, $\text{Span} (\left\{ z^{n} : n < 0\right\})$, the closure in $L^{2}(w)$ is $\overline{H^{2}_{0}}/\overline{\phi}$, where $\overline{H^{2}_{0}}$ is the orthogonal complement of $H^{2}$ in $L^{2}$. We shall be interested in the space
\begin{equation*}
  J_{\phi} = \frac{H^{2}}{\phi} \cap \frac{\overline{H^{2}_{0}}}{\overline{\phi}}.
\end{equation*}
It can be shown that $J_{\phi} = \left\{ 0 \right\}$ is equivalent to $\phi^{2}$ being an exposed point of the unit ball of $H^{1}$. Indeed from the current perspective it can be seen by observing that $\ker(T_{\overline{\phi}/\phi}) = \left\{ h \in H^{2} : h/\phi \in J_{\phi} \right\}$ and so $J_{\phi} = \left\{ 0 \right\}$ is equivalent to triviality of the kernel of $T_{\overline{\phi}/\phi}$. The result now follows from (X-2) in \cite{MR1289670}. The space $J_{\phi}$ consists of so-called pseudocontinuable functions. For $f \in J_{\phi}$ we can identify $f$ in the interior disk by its representation as a $H^{2}/\phi$ function and in the exterior disk, $\mathbb{D}^{e} = \left\{ z \in \mathbb{C}_{\infty} : \lvert z \rvert > 1 \right\}$, via the representation as a function in $\overline{H^{2}_{0}}/\overline{\phi}$. Thus $J_{\phi}$ consists of functions analytic in $C_{\infty} \setminus \mathbb{T}$, such that the non-tangential limits from outside and inside the unit disk coincide a.e. For definiteness we record the defining formulas for the extension of $f \in J_{\phi}$ to $\mathbb{C}_{\infty} \setminus \mathbb{T}$
\begin{equation*}
    f(z) = \phi(z)^{-1} \int_{\mathbb{T}} P(z, \zeta) \phi(\zeta)f(\zeta)dm(\zeta) \text{, for } z \in \mathbb{D},
\end{equation*}
and
\begin{equation*}
    f(z) = \overline{\phi}(1/\overline{z})^{-1} \int_{\mathbb{T}} P(1/\overline{z}, \zeta) \overline{\phi(\zeta)}f(\zeta)dm(\zeta) \text{, for } z \in \mathbb{D}^{e},
\end{equation*}
where $P(z,\zeta) = (1-\lvert z \rvert^{2})/\lvert z - \zeta\rvert^{2}$ is the Poisson kernel. The next theorem identifies $(aH^{2})^{\perp}$ as the space of normalized Cauchy transforms of $J_{\phi}$. We assume $b(0) = 0$ and $\mu_{1}$ is absolutely continuous.
\begin{theorem}\label{thm:model}
  Let $V = V_{1}: H^{2}/\phi \to \mathcal{H}(b)$ be the unitary map from Proposition \ref{prop:unitarymap}. The following holds:
  \begin{enumerate}[(i)]
    \item $V$ maps $J_{\phi}$ onto $(aH^{2})^{\perp}$,
    \item $V^{-1}M_{z}^{*}V = L$, where $Lf(z) = z^{-1}(f(z)-f(0))$.
  \end{enumerate}
  \begin{proof}
    For $\lambda \in \mathbb{D}$ we let $k_{\lambda}(z) = (1-\overline{\lambda}z)^{-1}$. For $f \in H^{2}$ there exists a unique function $g \in H^{2}/\phi$ with $af = Vg$. A simple computation based on Cauchy's formula gives $af = V(f/\bar{\phi})$. Hence
    \begin{equation*}
        0 = \int_{\mathbb{T}} \frac{f/\overline{\phi}-g}{1-\lambda \overline{\zeta}}\lvert \phi \rvert^{2}dm(\zeta) = \langle \phi (f/\overline{\phi} - g), \phi k_{\lambda} \rangle_{2},
    \end{equation*}
    for all $\lambda \in \mathbb{D}$. Since $\phi$ is outer the family $(\phi k_{\lambda})_{\lambda \in \mathbb{D}}$ has dense linear span in $H^{2}$ and thus $\phi f/\bar{\phi} - g \phi = v \in \overline{H^{2}_{0}}$. Rearranging we have $g = f/\bar{\phi} - v/\phi$. Then for $h \in H^{2}/\phi$ we have
    \begin{equation*}
      \langle Vh, af \rangle_{b} = \langle h, g \rangle_{\lvert \phi \rvert^{2}dm} = \int_{\mathbb{T}} (h\bar{\phi})\bar{f} dm - \int_{\mathbb{T}} h \phi \bar{v} dm = \int_{\mathbb{T}} (h\bar{\phi})\bar{f} dm,
    \end{equation*}
    The last integral is $0$ for all $f \in H^{2}$ if and only if $h\bar{\phi} \in \overline{H^{2}_{0}}$, which completes the proof of $(i)$. Part $(ii)$ is a special case of the identity $V^{-1}M_{U}^{*}V = T_{U}^{*}$ valid for any multiplier $U$ of $\mathcal{H}(b)$. Indeed by the proof of Proposition \ref{prop:unitarymap} we have the following
    \begin{equation*}
        \begin{split}
            M_{U}^{*}Vk_{\lambda} = (1-\overline{b(\lambda)})^{-1}M_{U}^{*}k_{\lambda}^{b} = (1-\overline{b(\lambda)})^{-1}\overline{U(\lambda)}k_{\lambda}^{b} \\ = \overline{U(\lambda)}Vk_{\lambda} = V\overline{U(\lambda)}k_{\lambda} = VT^{*}_{U}k_{\lambda} \text{, for each } \lambda \in \mathbb{D}.
        \end{split}
    \end{equation*}
    The result now follows from the completeness of the kernels.
  \end{proof}
\end{theorem}

\section{The spectrum of the adjoint of the shift}\label{sec:spec}

In this section, we choose the normalization $b(0) = 0$ and $\mu_{1}$ absolutely continuous, where $\mu_{1}$ is the Aleksandrov-Clark measure of $b$ at the point $1$. By Theorem \ref{thm:model} in the previous section the spectrum of $M_{z}^{*}$ restricted to $(aH^{2})^{\perp}$ is equal to the spectrum of the backwards shift $L$ on $J_{\phi}$, $\phi = a/(1-b)$. It turns out that the spectrum of $L$ on $J_{\phi}$ is related to analytic continuation of functions in $J_{\phi}$. This makes it easier to study the spectrum of $M_{z}^{*}$ by considering $L$ on $J_{\phi}$ instead. For $\lambda$ in the resolvent set of $L$ and $h \in J_{\phi}$ we have
\begin{equation}\label{eq:resolvent}
  (I - \lambda L)^{-1}h(z) = \frac{zh(z)- \lambda h(\lambda)}{z - \lambda}.
\end{equation}
Also if $k$ is the reproducing kernel at $0$, i.e. $h(0) = \langle h, k \rangle_{\phi}$, then
\begin{equation}\label{eq:fnc}
  h(\lambda) = \langle (I - \lambda L)^{-1}h, k \rangle_{\phi}.
\end{equation}
Actually, if all functions in $J_{\phi}$ are analytic at a point $\lambda$, then the resolvent is given by \eqref{eq:resolvent} (the uniform boundedness principle implies the point evaluation at $\lambda \in \mathbb{T}$ is bounded), conversely, if $\lambda \in \rho(L)^{-1}$, then every function in $J_{\phi}$ is analytic at $\lambda$ by \eqref{eq:fnc}. Thus $\sigma(L)^{-1}$ coincides with the set of points such that at least one function in $J_{\phi}$ does not extend analytically to $\lambda$. Clearly $\sigma(L)^{-1} \subset \mathbb{T}$ and so it equals $\overline{\sigma(L)}$.
\begin{definition}\label{def:specf}
  For an outer function $\phi \in H^{2}$ we denote by $\sigma(\phi)$ the set $\overline{\sigma(L)}$ constructed above. A point $e^{i\theta} \in \sigma(\phi)$ is called a point of local non-exposure for $\phi$.
\end{definition}

Clearly, for an outer function $\phi \in H^{2}$ we have $\sigma(\phi) = \varnothing$ if and only if $\phi^{2}$ is exposed, that is if it has no points of local non-exposure. Unfortunately, it seems to be very difficult in general to determine if a point $e^{i\theta} \in \mathbb{T}$ lies in $\sigma(\phi)$. Indeed this is equivalent to the problem of describing the exposed points of $H^{1}$. The remainder of this section is devoted to giving criteria for inclusion and exclusion in $\sigma(\phi)$.

If $\phi$ is ``not too small'' on an arc $I \subset \mathbb{T}$, then every function in $J_{\phi}$ extends analytically across $I$. The next proposition makes this statement precise.

\begin{proposition}\label{prop:spec_phi}
  Let $\phi$ be an outer function in $H^{2}$.
  \begin{enumerate}[(i)]
    \item If $\phi^{-1}$ is square summable on an arc $I \subset \mathbb{T}$, then $\sigma(\phi) \cap I = \varnothing$.
    \item If $\phi \in A = H^{\infty} \cap C(\mathbb{T})$, then $\sigma(\phi) \subset Z(\phi) = \left\{ \zeta \in \mathbb{D} \cup \mathbb{T} : \phi(\zeta) = 0 \right\}$.
  \end{enumerate}
  \begin{proof}
    Let $h \in J_{\phi}$, then $h = \phi h \phi^{-1}$ and hence $h$ is summable on $I$, which by Morera's theorem (see Ex. 2.12. in \cite{garnett2006bounded}) implies $h$ is analytic across $I$, this proves $(i)$. For $(ii)$ it suffices to notice that if $\phi(\zeta) \neq 0$ for some $\zeta \in \mathbb{T}$ then there exists $\eta > 0$ and an arc $\zeta \in I \subset \mathbb{T}$, such that $\lvert \phi \rvert > \eta$ on $I$ and apply Morera's theorem to $h = \phi h\phi^{-1}$, $h \in J_{\phi}$.
  \end{proof}
\end{proposition}

The inclusion $\sigma(\phi) \subset Z(\phi)$ can be proper as demonstrated by the function $\phi^{2} = (1-z) \in H^{1}$, which has a $0$ at $1$, but $\sigma(\phi) = \varnothing$ since it is easy to verify that $\phi^{2}$ is an exposed point.

We now turn to the relation between the Aleksandrov-Clark measures of $b$ and $\sigma(\phi)$. Recall for $\alpha \in \mathbb{T}$ that the Aleksandrov-Clark measures are Borel probability measures (we are still assuming $b(0) = 0$) given by
\begin{equation*}
  \frac{1 + \overline{\alpha} b(z)}{1 - \overline{\alpha}b(z) } = \int_{\mathbb{T}} \frac{\zeta+z}{\zeta-z}d\mu_{\alpha}(\zeta).
\end{equation*}
With this notation we have $d\mu_{1} = \lvert \phi \rvert^{2}dm$. By Theorem \ref{thm:poltoratski} the function $V_{\alpha}h$ converges non-tangentially to $h$ at $(\mu_{\alpha})_{s}$-almost every point and hence for an outer function $f \in \mathcal{H}(b)$ it is a necessary condition for cyclicity that $f$ be nonzero $(\mu_{\alpha})_{s}$-a.e. In light of this, it is natural to ask if this necessary condition is also sufficient. The answer to this question is negative. Indeed Poltoratski (see \cite{MR1889082}) has produced an example of an $\mathcal{H}(b)$ space such that all Aleksandrov-Clark measures $\mu_{\alpha}$ are absolutely continuous, but the function $a$ is not cyclic. However, see Proposition \ref{prop:1+theta/2} in Section \ref{sec:ex}. The next result clarifies the relationship between the support of the singular part of the Aleksandrov-Clark measure and the points of local non-exposure. We shall need the notion of Smirnov class and unbounded Toeplitz operators.

Let $N^{+}$ denote the Smirnov class of quotients of bounded analytic functions in $\mathbb{D}$ with outer denominator. Functions in $N^{+}$ have non-tangential limits a.e. and considering the boundary function we have the Smirnov maximum principle: $H^{p} = L^{p} \cap N^{+}$, for $0 < p \leq \infty$.

For a symbol $U \in L^{2}$ we define the unbounded Toeplitz operator $T_{U}$ by the rule
\begin{equation*}
  T_{U}h(z) = \int_{\mathbb{T}} \frac{U(\zeta)h(\zeta)dm(\zeta)}{1-z\overline{\zeta}}, \text{ for } h \in H^{2},
\end{equation*}
which clearly agrees with the usual definition if $U \in L^{\infty}$.

The proof of the next result is based on ideas from \cite{MR1207406}.
\begin{proposition}
  Let $\phi \in H^{2}$ be an outer function of unit norm and $(\mu_{\alpha})_{\alpha \in \mathbb{T}}$ the associated Aleksandrov-Clark measures. Then $\text{supp}(\mu_{\alpha})_{s} \subset \sigma(\phi)$.
  \begin{proof}
    Let $(\mu_{\alpha})_{s}$ be the singular part of the Aleksandrov-Clark measure $\mu_{\alpha}$. Without loss of generality, we may assume $\alpha \neq 1$ since $\mu_{1}$ is absolutely continuous. Let $\theta$ be the inner function defined by
    \begin{equation*}
      \frac{1 + \theta}{1 - \theta} = \int_{\mathbb{T}} \frac{\zeta + z}{\zeta - z}d(\mu_{\alpha})_{s}(\zeta).
    \end{equation*}
    The function $f = i(1+\theta)/(1-\theta) \in N^{+}$ and is real-valued on $\mathbb{T}$. Suppose, for a moment, that we can show $\phi/(1-\theta) \in H^{2}$, then $\phi f \in H^{2}$ and since $f$ is real-valued on $\mathbb{T}$ we also have $f \overline{\phi} \in \overline{H^{2}}$. Thus after subtracting an appropriate constant $f \in J_{\phi}$. Since $f$ has a singularity at every point in the support of $(\mu_{\alpha})_{s}$ the conclusion follows.

Thus it remains to show $\phi/(1-\theta) \in H^{2}$. It follows from the identity
  \begin{equation*}
    \frac{1-\overline{\alpha}b}{1-\theta} = (1-\overline{\alpha}b) \int_{\mathbb{T}} \frac{1}{1-z\overline{\zeta}}d(\mu_{\alpha})_{s}(\zeta),
  \end{equation*}
  that $(1-\overline{\alpha}b)/(1-\theta) \in \mathcal{H}(b)$. Hence there exists $h \in H^{2}/\phi$, such that $Vh = (1-\overline{\alpha}b)/(1-\theta)$. From this we see
  \begin{equation*}
    \frac{\phi}{1-\theta} = \frac{\phi Vh}{1-\overline{\alpha}b} = T_{\phi_{\alpha}}T_{\overline{\phi}}(\phi h),
  \end{equation*}
  where we recall that $\phi_{\alpha} = a/(1-\overline{\alpha}b)$. Thus if we can show that the operator $T_{\phi_{\alpha}}T_{\overline{\phi}}$ maps $H^{2}$ into $H^{2}$ we are done. Precisely this was shown in \cite{MR1207406}, see also (IV-16) in \cite{MR1289670}.
  \end{proof}
\end{proposition}

Poltoratski's example shows that the inclusion $\cup_{\alpha \in \mathbb{T}} \text{supp}(\mu_{\alpha})_{s} \subset \sigma(\phi)$ can be proper.

We end with a Proposition on the point spectrum of $M_{z}^{*}$ on $(aH^{2})^{\perp}$, which is certainly well known. However, we give a new short proof based on Theorem \ref{thm:model} to keep the discussion more self-contained. For $\gamma > 1$ and $\zeta \in \mathbb{T}$ we define the usual non-tangential cone
\begin{equation*}
    \Gamma_{\gamma}(\zeta) = \left\{ z \in \mathbb{D} : \lvert z - \zeta \rvert < \gamma(1-\lvert z \rvert)\right\}.
\end{equation*}

\begin{proposition}\label{prop:eigenvalue}
      Suppose $\zeta \in \mathbb{T}$ and $\overline{\zeta}$ is an eigenvalue of $M_{z}^{*} : (aH^{2})^{\perp} \to (aH^{2})^{\perp}$. Then 
      \begin{enumerate}[(i)]
        \item for each $f \in \mathcal{H}(b)$ the non-tangential limit
            \begin{equation*}
                \lim_{\substack{z \to \zeta \\ z \in \Gamma_{\gamma}(\zeta)}}f(z),
            \end{equation*}
            exists and is finite.
        \item The function $k_{\zeta}^{b}(z) = (1-\overline{b(\zeta)}b(z))/(1-\overline{\zeta}z)$ belongs to $\mathcal{H}(b)$ and for each $f \in \mathcal{H}(b)$ we have $f(\zeta) = \langle f, k_{\zeta}^{b} \rangle$, where $f(\zeta)$ is defined to be the non-tangential limit above.
      \end{enumerate}
      \begin{proof}
          Let $k_{\lambda}$ denote the Cauchy kernel. Since $\overline{\zeta}$ is an eigenvalue of $M_{z}^{*} : (aH^{2})^{\perp} \to (aH^{2})^{\perp}$ we deduce from Theorem \ref{thm:model} that $\overline{\zeta}$ is an eigenvalue of $L : J_{\phi} \to J_{\phi}$. A simple algebraic computation reveals that any eigenvector of $L$ corresponding to the eigenvalue $\overline{\zeta}$ must be a constant multiple of the Cauchy kernel $k_{\zeta}$. Hence $k_{\zeta} \in J_{\phi}$. For $\lambda \in \Gamma_{\gamma}(\zeta)$ and $z \in \mathbb{T}$ we have
          \begin{equation*}
              \gamma \lvert 1 - \overline{\lambda}z \rvert \geq \gamma (1-\lvert \lambda \rvert) > \lvert \lambda - \zeta \rvert  = \lvert \lambda - z + z - \zeta \rvert \geq \lvert z - \zeta \rvert - \lvert \lambda - z \rvert.
          \end{equation*}
          From this we deduce $(1+\gamma)\lvert 1-\overline{\lambda}z \rvert \geq \lvert 1 - \overline{\zeta}z \rvert$, for $\lambda \in \Gamma_{\gamma}(\zeta)$ and $z \in \mathbb{T}$. Hence $k_{\lambda}$ converges to $k_{\zeta}$ non-tangentially in the norm of $H^{2}/\phi$. We have previously seen that $Vk_{\lambda} = (1-b(\lambda))^{-1}k_{\lambda}^{b}$, for $\lambda \in \mathbb{D}$ and hence for $z \in \mathbb{D}$
          \begin{equation*}
            \frac{f(z)}{1-b(z)} = \langle f, (1-\overline{b(z)})^{-1}k_{z}^{b} \rangle_{b} = \langle VV^{-1}f,Vk_{z} \rangle_{b} = \langle V^{-1}f,k_{z} \rangle_{H^{2}/\phi}.
          \end{equation*}
          Thus $(1-b(z))^{-1}f$ converges non-tangentially as $z \to \zeta$ for each $f\in \mathcal{H}(b)$. Letting $f \equiv 1$ shows that $b$ converges non-tangentially at $\zeta$ and thus it follows that $f$ does as well for each $f \in \mathcal{H}(b)$.
      \end{proof}
\end{proposition}

\section{Proofs of the main results}\label{sec:suff}

In this section, we establish the main results stated in the introduction. We begin by giving a complete function theoretic characterization of cyclicty in the case $\dim(aH^{2})^{\perp} < \infty$ based on Proposition \ref{prop:ah2}.

\begin{proof}[Proof of Theorem \ref{thm:finite_defect}]
    Recall that we have assumed $\dim(aH^{2})^{\perp} < \infty$ and $\overline{\lambda_{1}}$, $\overline{\lambda_{2}}$, ... $\overline{\lambda_{s}}$ denotes the eigenvalues of $M_{z}^{*}$ restricted to $(aH^{2})^{\perp}$. Let $\mu = \mu_{1}$ be the Aleksandrov-Clark measure of $b$ associated to the point $1$. By a simple algebraic computation the eigenspaces of the backwards shift, $L$, on $P^{2}(\mu)$ are easily seen to be of dimension $1$. Since by Proposition \ref{prop:unitarymap} and part $(ii)$ of Theorem \ref{thm:model} the operator $M_{z}^{*}$ is unitarily equivalent to the backwards shift on $P^{2}(\mu)$ the eigenspaces of $M_{z}^{*}$ are also of dimension $1$.
    
    To see $(i)$ it suffices to note that if $\overline{\lambda} \in \mathbb{D}$ is an eigenvalue of $M_{z}^{*}$ restricted to $(aH^{2})^{\perp}$, then since the reproducing kernel $k_{\lambda}^{b}$ is an eigenvector it belongs to $(aH^{2})^{\perp}$. Thus we have
    \begin{equation*}
        0 = \langle a, k_{\lambda}^{b} \rangle_{b} = a(\lambda),    
    \end{equation*}
    contradicting that $a$ is outer. Statement $(ii)$ follows from Proposition \ref{prop:eigenvalue}. Also the kernel functions $k_{\lambda_{j}}^{b}$ belong to $\mathcal{H}(b)$ and $f(\lambda_{j}) = \langle f, k_{\lambda_{j}}^{b} \rangle$, for all $f \in \mathcal{H}(b)$. This takes care of the necessity in statement $(iii)$.

We now turn to sufficiency of $(iii)$. Suppose $f \in \mathcal{H}(b)$ is outer and $f(\lambda_{j}) \neq 0$, for all $j = 1, 2, ..., s$. Let $h \in \mathcal{H}(b)$ be arbitrary. It will suffice to show that $\langle z^{n}f, h \rangle = 0$, for all $n \geq 0$ implies $h \equiv 0$. We claim that we can assume $h \in (aH^{2})^{\perp}$, indeed decompose $h = h_{a} + \tilde{h}$, with $h_{a} \in \text{clos}(aH^{2})$ and $\tilde{h} \in (aH^{2})^{\perp}$. Proposition \ref{prop:ah2} implies that $aH^{2} \subset [f]$ and so $\langle ag, h_{a} \rangle = 0$, for all $g \in H^{2}$ and hence $h_{a} \equiv 0$. Let $n = \dim(aH^{2})^{\perp}$ and $m_{1}, m_{2}, ..., m_{s}$ be the algebraic multiplicities of the eigenvalues $\overline{\lambda}_{j}$, $j = 1, 2, ...,s$. We have $\sum_{j} m_{j} = n$. For each $1 \leq j \leq s$ and $1 \leq l \leq m_{j}$ choose and element $k_{\lambda_{j}}^{l}$ of norm $1$, such that
    \begin{equation*}
      k_{\lambda_{j}}^{l} \in \ker (\overline{\lambda}_{j} - M_{z}^{*})^{l},
    \end{equation*}
    and $k_{\lambda_{j}}^{l} \notin \ker (\overline{\lambda}_{j} - M_{z}^{*})^{k}$, for $1 \leq k < l$. This is possible by since, as noted earlier, $M_{z}^{*}$ has simple spectrum. Note that $k_{\lambda_{j}}^{1}$ is a constant multiple of the kernel element $k_{\lambda_{j}}^{b}$. Since $(aH^{2})^{\perp}$ is finite dimensional and $M_{z}^{*}$ acts on $(aH^{2})^{\perp}$ its root vectors (generalized eigenvectors) $(k_{\lambda_{j}}^{l})_{j,l}$ form a basis for $(aH^{2})^{\perp}$ (this is the Jordan normal form). Thus $h$ can be expressed
    \begin{equation*}
      h(z) = \sum_{j, l} c_{j,l}k_{\lambda_{j}}^{l}(z),
    \end{equation*}
    for some complex numbers $(c_{j,l})_{j,l}$. We will show $c_{j,l} = 0$ for all $j$ and $l$ which will complete the proof. For $1 \leq k \leq s$ let $p_{k}$ be the polynomial
    \begin{equation*}
      p_{k}(z) = (\lambda_{k} - z)^{m_{k}-1}\prod_{j \neq k}^{s} (\lambda_{j} - z)^{m_{j}}.
    \end{equation*}
    Then
    \begin{equation*}
      0 = \langle p_{k}f, h \rangle = \langle \prod_{j \neq k} (\lambda_{j} - z)^{m_{j}}f, (\overline{\lambda}_{k} - M_{z}^{*})^{m_{k}-1}c_{k,m_{k}}k_{\lambda_{k}}^{m_{k}}\rangle.
    \end{equation*}
    Since $(\overline{\lambda}_{k} - M_{z}^{*})^{m_{k}-1}k_{\lambda_{k}}^{m_{k}}$ is a nonzero element of $\ker(\overline{\lambda}_{k}-M_{z}^{*})$ it is of the form $a_{k}k_{\lambda_{k}}^{b}$, for some nonzero $a_{k}$. Thus
    \begin{equation*}
      0 = \langle \prod_{j \neq k} (\lambda_{j} - z)^{m_{j}}f, a_{k}c_{k,m_{k}}k_{\lambda_{k}}^{b}\rangle = \left( \prod_{j \neq k}(\lambda_{j} - \lambda_{k})^{m_{j}} \right)a_{k}c_{k,m_{k}}f(\lambda_{k}),
    \end{equation*}
    and so $c_{k,m_{k}} = 0$, for all $1 \leq k \leq s$. Replacing $m_{j}$ by $m_{j}-1$ (if $m_{j} > 1$) and iterating the above (finite) process shows $c_{j,l} = 0$ for all $1 \leq j \leq s$ and $1 \leq l \leq m_{j}$. Hence $h \equiv 0$.
\end{proof}

In the remainder of this section we focus on the case $\dim(aH^{2})^{\perp} = \infty$. Let $\phi \in H^{2}$ be an outer function of unit norm and consider the de Branges-Rovnyak space $\mathcal{H}(b)$ generated by $\phi$. Recall the heuristic principle stated in the introduction: an outer function $f \in \mathcal{H}(b)$ is cyclic if it is ``not too small'' on the set $\sigma(\phi)$. We turn to establishing the concrete realizations of this principle discussed in the introduction.

Recall that $N^{+}$ denotes the Smirnov class of quotients of bounded analytic functions in $\mathbb{D}$ with outer denominator. The proof of Theorem \ref{thm:thmA} is a direct consequence of the description of the invariant subspaces of $\mathcal{H}(b)$ in equation \eqref{eq:invariantsubspaces}.

\begin{proof}[Proof of Theorem \ref{thm:thmA}]
  Let $J$ be the map from Theorem \ref{thm:J_model} and $\psi$ the extremal function of $[f]$, that is
  \begin{equation*}
    [f] = \left\{ g \in \mathcal{H}(b) : \frac{g}{\psi}, \frac{g}{\psi}\psi_{1} \in H^{2} \right\},
  \end{equation*}
  where $J\psi = (\psi, \psi_{1})$. Recall that since $f$ is outer $aH^{2} \subset [f]$. Hence $ag \psi^{-1} \in H^{2}$ for all $g \in H^{2}$ which implies $a \psi^{-1} \in H^{\infty}$ and hence $\psi^{-1}\mathbbm{1}_{E} \in L^{2}$. Similarly $f \psi^{-1} \in L^{2}$ and so $\psi^{-1}\mathbbm{1}_{F} \in L^{2}$. Since $E \cup F = \mathbb{T}$ we see $\psi^{-1} \in L^{2}$ and hence since $\psi^{-1} \in N^{+}$ we have $\psi^{-1} \in H^{2}$. The same argument applies to $\psi^{-1}\psi_{1}$ and so $1 \in [f] = \mathcal{H}(b)$.
\end{proof}

Theorems \ref{thm:thmB} and \ref{thm:thmC} will require some preliminary lemmas.

\begin{lemma}\label{lemma:cauchymult}
  Let $\phi$ be an outer function in $H^{2}$, $g \in H^{2}/\phi$ and $h \in J_{\phi}$. Suppose $\langle g , L^{n}h \rangle_{\lvert \phi \rvert^{2}dm} = 0$ for all $n \geq 0$. Then for $\lvert \lambda \rvert < 1$
  \begin{equation*}
    \overline{h(1/\overline{\lambda})} \int_{\mathbb{T}} \frac{g\lvert \phi \rvert^{2}dm(\zeta)}{1 - \lambda \overline{\zeta}} = \int_{\mathbb{T}} \frac{\overline{h}g\lvert \phi \rvert^{2}dm(\zeta)}{1 - \lambda \overline{\zeta}}.
  \end{equation*}
  \begin{proof}
    Since $h \in J_{\phi}$ it has analytic psuedocontinuation to the exterior disk and it is in this sense that $h(1/\overline{\lambda})$ should be understood. For $\lvert w \rvert > 1$ we have by the assumption
    \begin{equation*}
      0 = \langle g, (I-wL)^{-1}Lh \rangle_{\lvert \phi \rvert^{2}dm} = \int_{\mathbb{T}} \overline{\left(\frac{h(\zeta)-h(w)}{\zeta-w}\right)}g\lvert \phi \rvert^{2}dm(\zeta).
    \end{equation*}
    Rearranging and using $\zeta-w = -w(1-\zeta/w)$ we have
    \begin{equation*}
      \overline{h(w)} \int_{\mathbb{T}} \frac{1}{1-\overline{\zeta}/\overline{w}}g\lvert \phi \rvert^{2}dm(\zeta)= \int_{\mathbb{T}} \frac{1}{1-\overline{\zeta}/\overline{w}}\overline{h}g\lvert \phi \rvert^{2}dm(\zeta).
    \end{equation*}
    The result now follows by setting $\lambda = 1/\overline{w}$.
  \end{proof}
\end{lemma}

We denote by $L^{1,\infty}_{0}$ the class of functions $h \in L^{1,\infty}$ satisfying
\begin{equation*}
  m( \left\{ e^{i\theta} \in \mathbb{T} : \lvert h(e^{i\theta}) \rvert > t \right\} ) = o(1/t), t \to \infty.
\end{equation*}
The set $L^{1,\infty}_{0}$ is a closed subspace of $L^{1,\infty}$. We define its analytic subspace $H^{1,\infty}_{0} = N^{+} \cap L^{1,\infty}_{0}$. It is a result of Kolmogorov that for $h \in L^{1}$ the Cauchy integral $Ch$ belongs to $H^{1,\infty}_{0}$, $Ch \in H^{1,\infty}_{0}$. The next Lemma is essentially due to Aleksandrov. Indeed the it is a direct Corollary to Theorem 6 of \cite{MR627941}, see also Lemma 5.13 and Lemma 5.22 in \cite{MR3589672}.

\begin{lemma}\label{lemma:weakmorera}
  Suppose $f, \overline{f} \in H^{1,\infty}_{0}$. Then $f$ is constant.
\end{lemma}

With this in hand, we are ready for the proof of Theorem \ref{thm:thmB}.

\begin{proof}[Proof of Theorem \ref{thm:thmB}]
    Recall that $f \in \mathcal{H}(b)$ is an outer function, such that for each point $\zeta \in \sigma(\phi)$ there exists an open arc $\zeta \in I_{\zeta} \subset \mathbb{T}$ and a number $\eta_{\zeta} > 0$, such that $\lvert f \rvert > \eta_{\zeta}$ a.e. on $I_{\zeta}$. Let $V=V_{1}: P^{2}(\mu_{1}) \to \mathcal{H}(b)$ be the unitary map from Proposition \ref{prop:unitarymap} and recall that we can without loss of generality assume that $\phi = a/(1-b)$ is of unit norm and hence $P^{2}(\mu_{1}) = H^{2}/\phi$. Suppose $\langle z^{n}f, Vh \rangle_{b} = 0$ for all $n \geq 0$ and some $h \in H^{2}/\phi$. It will be sufficient to show $h \equiv 0$. Since $\text{clos}(aH^{2}) \subset [f]$ we can without loss of generality restrict to the case $Vh \in (aH^{2})^{\perp}$ or equivalently $h \in J_{\phi}$. Let $g \in H^{2}/\phi$ the unique function satisfying $Vg = f$. By Lemma \ref{lemma:cauchymult} we have for $\lvert \lambda \rvert < 1$
    \begin{equation*}
      \overline{h(1/\overline{\lambda})}f(\lambda) = (1-b(\lambda))\int_{\mathbb{T}} \frac{\overline{h}g\lvert\phi\rvert^{2}}{1-\lambda\overline{\zeta}}dm(\zeta).
    \end{equation*}
    Since $\overline{h}g\lvert\phi\rvert^{2} \in L^{1}$ we have $\overline{h}f \in L^{1, \infty}_{0}$. The set $\sigma(\phi)$ is compact, hence there exists a number $\eta > 0$ and a finite collection of $I_{\zeta_{j}}$, $j = 1, 2, ..., N$ such that $\sigma(\phi) \subset \cup_{j=1}^{N} I_{\zeta_{j}} = I$, $\lvert f \rvert > \eta > 0$ a.e. on $I$ and $h$ is analytic in a neighborhood of $\mathbb{T} \setminus I$. From this we see $\overline{h} \in L^{1, \infty}_{0}$. Recall that $J_{\phi} \subset N^{+}$ and so $h, \overline{h} \in N^{+}$. Hence $h, \overline{h} \in H^{1, \infty}_{0}$. Since $h(\infty) = 0$ we have $h \equiv 0$ by Lemma \ref{lemma:weakmorera}.
  \end{proof}

With the above tools in hand, we can prove the following necessity theorem for annihilators of the *-cyclic subspace generated by $h \in J_{\phi}$.

\begin{proposition}
  Let $\mathcal{H}(b)$ be a non-extreme de Branges-Rovnyak space, with $b(0) = 0$. We can without loss of generality assume that $\phi = a/(1-b)$ is of unit norm. Let $g \in H^{\infty}$ and write $f = Vg$, where $V$ is the unitary map from Proposition \ref{prop:unitarymap}. Denote by $\theta$ the inner factor of $f$ and set $F = f/\theta$. Suppose $h \in J_{\phi} \setminus \left\{ 0 \right\}$ and $g$ annihilates the *-cyclic subspace generated by $h$, that is $\langle g, L^{n}h \rangle_{J_{\phi}} = 0$, for all $n \geq 0$. Then $h$ has a singularity at $\sigma(F)$.
  \begin{proof}
    From Lemma \ref{lemma:cauchymult} we have for $\lvert \lambda \rvert < 1$
    \begin{equation*}
      \overline{h(1/\overline{\lambda})}f(\lambda) = (1-b(\lambda))\int_{\mathbb{T}} \frac{\overline{h(\zeta)}g(\zeta)\lvert \phi \rvert^{2}dm(\zeta)}{1-\lambda \overline{\zeta}}.
    \end{equation*}
    Since $g$ is bounded $\overline{h}g \in H^{2}/\phi$ and hence $\overline{h}f \in H^{2}$. Since $\theta$ is the inner factor of $f$ and $\overline{h}(0) = 0$ we have $\overline{h}F = \overline{h}f/\theta \in H^{2}_{0}$. Since $hF \in N^{+}$ is also square summable on $\mathbb{T}$ we have
    \begin{equation*}
      h \in \frac{H^{2}}{F} \cap \frac{\overline{H^{2}_{0}}}{\overline{F}} = J_{F}.
    \end{equation*}
    Thus $h$ has a singularity at $\sigma(F)$.
  \end{proof}
\end{proposition}

Theorem \ref{thm:thmC} follows immediately from the above result.

\section{Examples}\label{sec:ex}

In this section, we apply the main results to give examples of cyclic vectors in certain $\mathcal{H}(b)$ spaces. We begin with a result valid for all non-extreme de Branges-Rovnyak spaces.

\begin{proposition}
    Let $b$ be a non-extreme function in the unit ball of $H^{\infty}$ and $\mathcal{H}(b)$ the corresponding de Branges-Rovnyak space. Then
    \begin{enumerate}[(i)]
        \item For each $\lambda \in \mathbb{D}$ the kernel $k_{\lambda}^{b}(z) = (1-\overline{b(\lambda)}b(z))/(1-\overline{\lambda}z)$ is a cyclic vector.
        \item The function $b$ is cyclic if and only if it is outer.
    \end{enumerate}
    \begin{proof}
        Part $(i)$ can be seen by appealing to Theorem \ref{thm:thmC}. Indeed the Cauchy kernel $k_{\lambda}$ belongs to $H^{\infty}$ and $k_{\lambda}^{b} = (1-\overline{b(\lambda)})Vk_{\lambda}$. Since $\lvert k_{\lambda}^{b}(\zeta) \rvert \geq 2^{-1}\lvert 1 - \lvert b(\lambda) \rvert > 0$ for almost every $\zeta \in \mathbb{T}$ we have $J_{k_{\lambda}^{b}} = \varnothing$, by Proposition \ref{prop:spec_phi}.
        
        To see part $(ii)$ it suffices to notice that the identity $\lvert a \rvert^{2} + \lvert b \rvert^{2} = 1$ valid a.e. on $\mathbb{T}$ prevents $a$ and $b$ from being small simultaneously (up to a set of measure $0$) and hence by Theorem \ref{thm:thmA} the function $b$ is a cyclic vector if it is outer.
    \end{proof}
\end{proposition}

    Theorem \ref{thm:thmB} is useful for studying cyclic vectors in $\mathcal{H}(b)$ spaces, where $\phi = a/(1-b)$ can be written as $\phi = F\phi_{1}$ and $F^{2}$ is, after a possible normalization, an exposed point of the unit ball of $H^{1}$. In this case one can often discard $F$ as a "trivial factor". We give an example where the function $\phi_{1}$ has only one point of local non-exposure.

    We say that a function a measurable function $f : \mathbb{T} \to \mathbb{C}$ is separated from $0$ at a point $\zeta \in \mathbb{T}$ if there exists a constant $\epsilon$ and an open arc $\zeta \in I \subset \mathbb{T}$, such that $\lvert f \rvert > \epsilon$ a.e. on $I$.
    
\begin{proposition}
    Let $\phi = F \phi_{1} \in H^{2}$ be an outer function of unit norm and $\mathcal{H}(b)$ be the de Branges-Rovnyak space generated by $\phi$. Let, in addition, $\phi_{1}, F \in H^{2}$ be outer functions and suppose $F^{2}/\| F^{2} \|_{1}$ is an exposed point of the unit ball of $H^{1}$. Suppose that for each open arc $I \subset \mathbb{T}$ containing $1$ there exists a constant $\epsilon = \epsilon(I) > 0$, such that $\lvert \phi_{1} \rvert > \epsilon$ a.e. on $\mathbb{T} \setminus I$. Then $\sigma(\phi) \subset \left\{ 1 \right\}$. Moreover, if $f \in \mathcal{H}(b)$ is an outer function and $f$ is separated from $0$ at the point $1$, then $f$ is a cyclic vector in $\mathcal{H}(b)$.
    \begin{proof}
        Let $h \in J_{\phi} \setminus \left\{ 0 \right\}$ and suppose for a contradiction that $h$ is analytic across some open arc $I$ containing $1$, then
        \begin{equation*}
            Fh = \mathbbm{1}_{\mathbb{T} \setminus I} (F\phi_{1}h)\phi_{1}^{-1} + \mathbbm{1}_{I} Fh.
        \end{equation*}
        Since $\phi_{1}^{-1} \in L^{\infty}(\mathbb{T} \setminus I)$ and $h$ is analytic across $I$ we have $Fh \in L^{2}$. Since also $h\overline{F} = \overline{F\phi_{1}}h\overline{\phi_{1}}^{-1} = \overline{\phi}h\overline{\phi_{1}}^{-1} \in \overline{N^{+}}$ it follows that
        \begin{equation*}
            h \in \frac{H^{2}}{F} \cap \frac{\overline{H^{2}}}{\overline{F}}.
        \end{equation*}
        Since $F^{2}$ is an exposed point of $H^{1}$ this implies $h \equiv 0$ giving a contradiction. In particular, each function in $J_{\phi} \setminus \left\{ 0 \right\}$ must have a singularity at the point $1$. We show that this implies $\sigma(\phi) \subset \left\{ 1 \right\}$. Suppose for a contradiction $\zeta \neq 1$ and $\zeta \in \sigma(\phi)$. By Theorem 3 in \cite{MR1734331} there exists a function $h \in J_{\phi}$, such that all of its singularities are contained inside an open arc $\zeta \in U \subset \mathbb{T}$ and $1 \notin U$ contradicting that $h$ must have a singularity at $1$. The second part of the Proposition follows from Theorem \ref{thm:thmB} and the fact that $\sigma(\phi) \subset \left\{ 1 \right\}$.
    \end{proof}
\end{proposition}

Let $A = C(\mathbb{T}) \cap H^{\infty}$ be the disk algebra. For $f \in A$ we let $Z(f) = \left\{ \zeta \in \mathbb{D} \cup \mathbb{T} : f(\zeta) = 0 \right\}$ denote its zero set.

\begin{proposition}
    Let $f \in \mathcal{H}(b) \cap A$ be an outer function. If $Z(f) \cap \sigma(\phi) = \varnothing$, then $f$ is cyclic.
    \begin{proof}
        By Theorem \ref{thm:thmB} it suffices to show that $f$ is separated from zero in a neighborhood of each point $\zeta \in \sigma(\phi)$. By assumption $f$ is continuous and non-zero on $\sigma(\phi)$ and hence it follows immediately.
    \end{proof}
\end{proposition}

Let us now give an application of our results to generalized Dirichlet spaces. Let $\mu$ be a positive finitely supported measure on $\mathbb{T}$. The Dirichlet space $\mathcal{D}(\mu)$ associated with $\mu$ is the set of $f \in H^{2}$ that have non-tangential limits $\mu$-a.e. and such that the generalized Dirichlet integral
\begin{equation*}
  \mathcal{D}_{\mu}(f) = \int_{\mathbb{T}} \int_{\mathbb{T}} \left\lvert \frac{f(z) - f(w)}{z-w} \right\rvert d\mu(w)dm(z),
\end{equation*}
is finite. The norm on $\mathcal{D}(\mu)$ is given by $\| f\|^{2} = D_{\mu}(f) + \| f \|_{2}^{2}$. In \cite{MR3110499} it was shown that $\mathcal{D}(\mu) = \mathcal{H}(b)$ with equivalent norms for some polynomials $b$ and $a$ where $a$ has a simple zero at each point in the support of $\mu$ and no other zeros. Combining this with Theorem \ref{thm:finite_defect} result yields the following Proposition which was obtained directly in \cite{MR3000683}.

\begin{proposition}
  Let $\mu$ be a finitely supported measure on $\mathbb{T}$ and $\mathcal{D}(\mu)$ the associated Dirichlet space. Denote by $\lambda_{1}$, $\lambda_{2}$, ..., $\lambda_{s}$ the support of $\mu$. Then $f \in \mathcal{D}(\mu)$ is cyclic if and only if $f$ is outer and $f(\lambda_{j}) \neq 0$, for all $j = 1, 2, ..., s$.
\end{proposition}

We now consider a case with $\text{dim}(aH^{2})^{\perp} = \infty$ where we can give simple necessary and sufficient conditions for cyclicity. Let $\theta$ be a non-constant inner function and $b = (1+\theta)/2$. In the literature the choice $\theta = z$ is common. We remark that in this case, one does not have $b(0) = 0$. One can consider an equivalent version with $b(0) = 0$, however, to stay consistent with the literature we prefer to consider the non-normalized version. Since $b(0) \neq 0$ in the next Proposition it is important to recall that the Aleksandrov-Clark measures will not be probability measures in this case.

For an inner function $\theta$ we define the usual model space $K_{\theta} = H^{2} \cap \theta \overline{H^{2}_{0}}$.

\begin{proposition}\label{prop:1+theta/2}
    Let $\theta$ be a non-constant inner function and $b = (1+\theta)/2$. Then
    \begin{enumerate}[(i)]
        \item $\mathcal{H}(b) = \frac{1}{2}(1-\theta)H^{2} \oplus K_{\theta}$, where the orthogonal sum is taken with respect to the norm in $\mathcal{H}(b)$.
        \item All functions in $\mathcal{H}(b)$ have non-tangential limits $\sigma$-a.e., where $\sigma$ is the measure, such that
            \begin{equation*}
                \frac{1 - \lvert \theta \rvert^{2}}{\lvert 1 -\theta \rvert^{2}} = \int_{\mathbb{T}} \frac{1-\lvert z \rvert^{2}}{\lvert \zeta - z\rvert^{2}}d\sigma(\zeta).
            \end{equation*}
        \item Let $f \in \mathcal{H}(b)$ be an outer function. Then $f$ is cyclic if and only if $f$ is non-zero $\sigma$-a.e.
    \end{enumerate}
    \begin{proof}
        A simple computation gives
        \begin{equation*}
            \frac{1+b}{1-b} = \frac{3+\theta}{1-\theta} = 1 + 2\frac{1+\theta}{1-\theta}.
        \end{equation*}
        Taking real parts we see
        \begin{equation*}
            \frac{1 - \lvert b \rvert^{2}}{\lvert 1-b \rvert^{2}} = \int_{\mathbb{T}} \frac{1-\lvert z \rvert^{2}}{\lvert \zeta - z\rvert^{2}} d(m+2\sigma)(\zeta).
        \end{equation*}
        Thus $\mu = m + 2\sigma$ is the Aleksandrov-Clark measure of $b$ associated to the point $1$. Note that $\mu$ is not absolutely continuous and not of total mass $1$. It is easy to see that $a = (1-\theta)/2$, hence $\phi = a/(1-b) = 1$ and $J_{\phi} = \varnothing$. It follows from the Remark following Proposition \ref{prop:unitarymap} that
        \begin{equation*}
            \mathcal{H}(b) = VP^{2}(m+2\sigma) = VH^{2} \oplus VL^{2}(2\sigma) = \frac{1-\theta}{2}H^{2} \oplus K_{\theta}.
        \end{equation*}
        Thus part $(i)$ is proved. Part $(ii)$ follows from Poltoratski's Theorem on boundary convergence of normalized Cauchy transforms. Now let $f \in \mathcal{H}(b)$ be an outer function which is non-zero $\sigma$-a.e. Let $Vh \in \mathcal{H}(b)$ and suppose $\langle z^{n}f,Vh \rangle = 0$ for all $n \geq 0$. We must show $Vh \equiv 0$. We may suppose $Vh \in (aH^{2})^{\perp}$. Since $aH^{2} = \frac{1-\theta}{2}H^{2}$ it must be that $(aH^{2})^{\perp} = VL^{2}(2\sigma) = K_{\theta}$, and thus $Vh \in (aH^{2})^{\perp}$ is equivalent to $h \in L^{2}(\sigma)$. Since $f$ converges to $V^{-1}f$ non-tangentially $\sigma$-a.e. we see that $z^{n}f$ converges non-tangentially to $\zeta^{n}V^{-1}f$ $\sigma$-a.e. Thus
        \begin{equation*}
            0 = \langle z^{n}f, Vh \rangle_{b} = \int_{\mathbb{T}} \zeta^{n}V^{-1}f\overline{h}d\sigma(\zeta) \text{, for all } n \geq 0.
        \end{equation*}
        Since polynomials are dense in $L^{2}(\sigma)$ (the measure is singular) we have that $V^{-1}f\overline{h} = 0$ $\sigma$-a.e. Since $V^{-1}f$ is non-zero $\sigma$-a.e. we see $h = 0$ $\sigma$-a.e. and thus $Vh \equiv 0$. Conversely, if $V^{-1}f$ is not non-zero $\sigma$-a.e. it follows from the above computation that $f$ cannot be cyclic and hence the result is proved.
    \end{proof}
\end{proposition}

\bibliography{citations}

\end{document}